\documentclass[leqno,12pt]{amsart} 
\usepackage{amssymb} 
\usepackage{braket}
\usepackage{bm}

\theoremstyle{plain} 
\newtheorem{theorem}{\noindent\bf Theorem}[section] 
\newtheorem{lemma}[theorem]{\noindent\bf Lemma}
\newtheorem{corollary}[theorem]{\noindent\bf Corollary}
\newtheorem{proposition}[theorem]{\noindent\bf Proposition}

\theoremstyle{definition} 
\newtheorem{remark}[theorem]{\noindent\it Remark}

\newcommand{\Ker}[0]{\operatorname{Ker}}

\newcommand{\deldel}{\sqrt{-1}\partial \overline{\partial}}
\newcommand{\del}{\sqrt{-1}\partial}
\newcommand{\dbar}{\overline{\partial}}
\newcommand{\e}{\varepsilon}

\newcommand{\lla}[0]{{\langle\!\hspace{0.02cm} \!\langle}}
\newcommand{\rra}[0]{{\rangle\!\hspace{0.02cm}\!\rangle}}

\newcommand{\ke}{\omega_{\rm KE}}

\newcommand{\calpha}{(\alpha_i)_{i=1}^N}
\newcommand{\comega}{(\omega_i)_{i=1}^N}
\newcommand{\ctheta}{(\theta_i)_{i=1}^N}
\begin{document}

\title[Deformation for coupled K\"ahler-Einstein metrics]{Deformation for coupled K\"ahler-Einstein metrics} 
\author[S. Nakamura]{Satoshi Nakamura} 

%\author[S. Author]{Kaho Harigaya} 
%\dedicatory{Dedicated to Professor Xxx Yyy on his sixtieth birthday}

\thanks{ %Thanks
}

\address{
Department of Mathematics\\
Saitama University\\ 
255 Shimo-Okubo, Sakura-Ku, Saitama 380-8570\\ 
Japan
}
\email{nakamura1990@mail.saitama-u.ac.jp}

%\address{%Mathematical Institute \endgraf
%Tohoku University \endgraf
%Sendai 980-8578 \endgraf
%Japan
%}
%\email{author2@math.tohoku.ac.jp}

%%%%%%%%%%%%%%%%%%%%%%%%%%%%%%%%%%%%%%%%%

\maketitle

\begin{abstract}
The notion of coupled K\"ahler-Einstein metrics was introduced recently by Hultgren-WittNystr\"om.
In this paper we discuss deformation of a coupled K\"ahler-Einstein metrics on a Fano manifold.
We obtain a necessary and sufficient condition for a coupled K\"ahler-Einstein metric to be deformed to another coupled K\"ahler-Einstein metric for a Fano manifold admitting non-trivial holomorphic vector fields.
%This generalizes Hultgren-WittNystr\"om's result.
In addition we also discuss deformation for a coupled K\"aher-Einstein metric on a Fano manifold when the complex structure varies. 
\end{abstract}

%\newpage
%\tableofcontents

%%%%%%%%%%%%%%%%%%%%%%%%%%%%%%%%%%%%%
\section{Introduction}

The existence problem for canonical K\"ahler metrics for polaraized manifolds is one of the central topics in K\"ahler geometry.
In particular, K\"ahler-Einstein metrics for Fano manifolds has been discussed by many experts from 1980's.
In 2015, Chen-Donanldson-Sun \cite{CDS1, CDS2, CDS3} and Tian \cite{Tia15} showed that a Fano manifold admits a K\"ahler-Einstein metrics if and only if it satisfies K-polystablility which is an algebro-geometric condition originates on geometric invariant theory.

Some generalizations of K\"ahler-Einstein metrics for Fano manifolds has also been discussed.
In this paper, we focus on coupled K\"ahler-Einstein metrics introduced recently by Hultgre-WittNystr\"om \cite{HN18}.
Let $X$ be an $n$-dimensional Fano manifold.
A decomposition of the first Chern class $2\pi c_1(X)$ is a sum
$$2\pi c_1(X)=\alpha_1 + \cdots + \alpha_N$$
where each $\alpha_i$ is a K\"ahler class.
For a K\"ahler metric $\omega_i\in \alpha_i$, the pair $\comega$ is called the {\it coupled K\"ahler-Einstein metric} for the decomposition $\calpha$ if it satisfies the system
$$\mathrm{Ric}(\omega_1)=\cdots = \mathrm{Ric}(\omega_N)=\sum_{i=1}^{N}\omega_i.$$
Note that the ordinary K\"ahler-Einstein metric $\omega_{\rm KE}\in2\pi c_1(X)$ can be seen as a trivial coupled K\"ahler-Einsten metric. 
Indeed, for fixed $\lambda_i \in \mathbb{R}_{>0}$ satisfying $\sum_{i=1}^N\lambda_i=1$, the pair $(\lambda_i\omega_{\rm KE})_{i=1}^N$ defines a coupled K\"ahler-Einstein metric for the decomposition $(\lambda_i 2\pi c_1(X))_{i=1}^N$.

Coupled K\"ahler-Einstein metrics does not always exist for any decomposition $\calpha$.
In fact, the Matsushima type obstruction theorem and the Futaki type obstruction theorem
 %\cite{HN18} (see also \cite{FZ18, DP19}) for $(\alpha_i)_i$ 
 are known.
 The Matsushima type obstruction theorem states that if $\comega$ is a coupled K\"ahler-Einstein metric, then the identity component of the holomorphic automorphism group ${\rm Aut}_0(X)$ is the complexification of the identity component of the isometry group ${\rm Isom}_0(X, \omega_1)$ of $\omega_1$, 
 and in particular, it is then reductive \cite{HN18, FZ18}.
 Note that, in this case, ${\rm Isom}_0(X, \omega_1)={\rm Isom}_0(X, \omega_2)=\cdots ={\rm Isom}_0(X, \omega_N)$ as pointed out in \cite{Tak19}.
On the other hand, for the Futaki type obstruction, the Futaki type invariant is defined in the following way.
First define a function $f_i \in C^{\infty}(X, \mathbb{R})$ for each $i=1,2, \dots , N$ as
\begin{equation}\label{Ricci}
{\rm Ric}(\omega_i)-\sum_{j=1}^N\omega_j = \deldel f_i \quad\text{and}\quad \int_X(1-e^{f_i})\omega_i^n=0.
\end{equation}
In this paper, we call the pair $(f_i)_{i=1}^N$ {\it the Ricci potential} for $\comega$.
Note that the pair $\comega$ defines a coupled K\"ahler Einstein metric if and only if every Ricci potential $f_i$ vanishes.
For any holomorphic vector field $V$ on $X$, {\it the Futaki type invariant} ${\rm Fut}_c$ for the decomposition $\calpha$ is defined by
\begin{equation}
{\rm Fut}_c(V)=\sum_{i=1}^N\int_X H_i(1-e^{f_i})\frac{\omega_i^n}{\int_X\omega_i^n}
\end{equation}
where $H_i$ is the potential function for $V$, that is, $i_V\omega_i=\sqrt{-1}\dbar H_i$.
Following \cite{HN18, FZ18, DP19}, the value ${\rm Fut}_c(V)$ is independent of the choice of metrics $\omega_i\in\alpha_i$.
In particular, if there exists a coupled K\"ahler Einstein metric for $\calpha$, then the invariant ${\rm Fut}_c$ must vanish identically.
Recently Futaki-Zhang \cite{FZ19} showed a residue formula to compute this invariant.

In the same spirit of theory for ordinary K\"ahler-Einstein metrics for Fano manifolds, Hultgren-WittNystr\"om \cite{HN18} established fundamental properties for coupled K\"ahler-Einstein metrics.
It was shown a uniqueness theorem, 
and a stability theorem which states that the existence for a coupled K\"ahler-Einstein metric implies an algebro-geometric stability condition when $\alpha_i=2\pi c_1(L_i)$ for some ample line bundle $L_i$ over $X$.
%As the existence theorem, it was shown that ordinary K\"ahler-Einstein-metrics for Fano manifolds without non-trivial holomorphic vector fields can be deformed to a non-trivial coupled K\"ahler-Einstein metric.
%As explained later, the main result of this paper is a generalization of this existence theorem.
The work of Hultgren-WittNystr\"om has raised much interest in the study of coupled K\"ahler-Einstein metrics.
In particular, the existence theorem were developed.
Pingali \cite{Pin18} and Takahashi \cite{Tak19} introduced a continuity method and a Ricci iteration method respectively to construct coupled K\"ahler-Einstein metrics.
Hutgren \cite{Hul17} developed a detailed study for the existence of such metrics on toric Fano manifolds, and Delcroix-Hultgren \cite{DH18} extended it to more general settings.
Futaki-Zhang \cite{FZ18} introduced Sasakian analogue.
Datar-Pingali \cite{DP19} introduced the notion of coupled constant scalar curvature K\"ahler metrics which is a generalization of coupled K\"ahler-Einstein metrics, and also introduced a framework of geometric invariant theory for them.
Takahashi \cite{Tak19-2} introduced a geometric quantization.% for coupled K\"ahler-Einstein metrics.

%\subsection{Main result}
%We consider  deformation of coupled K\"ahler-Einstein metrics by following a strategy of deformation theory for constant scalar curvature K\"ahler metrics on K\"ahler manifolds \cite{LS94, HLi17}.
%The first result is to obtain a necessary and sufficient condition for an initial coulpled K\"ahler-Einstein metric to be deformed to another one for another 
%close decomposition for Fano manifolds admitting non-trivial holomorphic vector fields.
Now we state results in this paper.
As an existence theorem for non-trivial coupled K\"ahler-Einstein metrics, Hultgren-WittNystr\"om \cite{HN18} proved the following;
\begin{theorem}{\rm \cite[Theorem C]{HN18}}\label{HN}
Let $X$ be a Fano K\"ahler-Einstein manifold without non-trivial holomorphic vector fields.
Fix positive real constants $\lambda_i>0$ satisfying $\sum_{i=1}^N\lambda_i=1$.
Then for any real closed $(1,1)$-forms $\eta_1, \dots,\eta_N$ satisfying $\sum_{i=1}^N[\eta_i]=0$, there exists a coupled K\"ahler-Einstein metric for the decomposition $(2\pi \lambda_i c_1(X)+t[\eta_i])_{i=1}^N$ for $0<t \ll 1$.
\end{theorem}
The main purpose of this paper is to extend the above theorem from view points of (i) the case when a Fano manifold admits continuous automorphism group and (ii) the case of deformation for a non-trivial coupled K\"ahler-Einstein metric.
We follow the strategy of deformation theory for constant scalar curvature K\"ahler metrics \cite{LS94, HLi17} and extremal K\"ahler metrics \cite{RST13}.
Let $X$ be a Fano manifold admitting a coupled K\"ahler-Einstein metric $\ctheta$ for a decomposition $(\alpha_i)_{i=1}^N$. 
%and $G$ be the isometry group ${\rm Isom}(X, \theta_1)$.
Take a K\"ahler metric $\omega_0$ defined by $\mathrm{Ric}(\omega_0)=\sum_{i=1}^N\theta_i$,
We define
\begin{equation}
U=\Set{\eta=(\eta_1, \dots, \eta_N) | \eta_i \text{ is a $\mathbb{R}$-valued $\theta_i$-harmonic $(1,1)$-form s.t. $\sum_{i=1}^N [\eta_i]=0$}}.
\end{equation}
%Take a K\"ahler metric $\omega_0$ defined by $\mathrm{Ric}(\omega_0)=\sum_{i=1}^N\theta_i$,
and define  $U_0=\set{\eta\in U | \|\eta\|_{\omega_0}=1}$.
The following is the main result of this paper.

\begin{theorem}\label{main thm}
There exists $\e_0>0$ and a smooth function $\mathcal{F}: [0, \e_0)\times U_0 \to\mathbb{R}$
such that if $\eta\in U_0$ satisfies $\mathcal{F}(t, \eta)=0$ for some $t\in [0, \e_0)$, then there exists a coupled K\"ahler-Einstein metric for the decomposition $(\alpha_i+t[\eta_i])_{i=1}^N$.

Moreover, if $\mathrm{tr}_{\theta_i}\eta_i=0$ for all $i$, 
then the leading term of the asymptotic expansion of the function $\mathcal{F}(t,\eta)$ around $t=0$ is at least of order $2$.
%and its leading coefficient is described in terms of initial data.
If  $\mathrm{tr}_{\theta_i}\eta_i=0$ for all $i$ 
and if $\ctheta$ is a trivial coupled K\"ahler-Einstein metric $(\lambda_i\ke)_{i=1}^N$, 
then the leading term is at least of order $4$.
Furthermore these leading coefficient are described explicitly in terms of initial data
%Moreover the function $\mathcal{F}$ has an asymptotic expansion of order $2$ at $t=0$ in general. 
%In the case when $\ctheta$ is a trivial coupled K\"ahler-Einstein metric $(\lambda_i\ke)_{i=1}^N$, the asymptotic expansion is of order $4$.
%These leading terms are obtained explicity in terms of initial data 
(See section \ref{expansion} for the explicit description of these leading coefficients).
\end{theorem}

This theorem is divided to Theorem \ref{part1}, Proposition \ref{part2} and Proposition \ref{part3}.
The function $\mathcal{F}$ in Theorem\ref{main thm} tells us in which directions we can find a coupled K\"ahler-Einstein metric.
The function $\mathcal{F}$ is in fact given by the Futaki type invariant ${\rm Fut}_c(V_{t, \eta})$ 
for the decomposition $(\theta_i+t[\eta_i])_{i=1}^N$, where $V_{t, \eta}$ is a holomorphic vector field depending on $t\in [0, \e_0)$ and $\eta \in U_0$.
Therefore Theorem\ref{main thm} is a generalization of Theorem \ref{HN}.

In view of Theorem \ref{main thm}, it is natural to discuss the case when $\mathcal{F}(t, \eta) \neq 0$ for some small $t \neq 0$. 
%but this function has only an asymptotic expansion of some order at $t=0$.
%In Theorem\ref{main thm}, if $\mathcal{F}(t, \eta) \neq 0$, then there exists no coupled K\"ahler Einstein metrics for $(\lambda_i[\ke]+t[\eta_i])_i$.
%However,  if it has an asymptotic expansion $\mathcal{F}(t, \eta)=O(t^m)$ for some $m\in\mathbb{N}$, 
In this case there exists no coupled K\"ahler-Einstein metric for the decomposition $(\alpha_i+t[\eta_i])_{i=1}^N$.
However, under the assumption that the function $\mathcal{F}$ has an asymptotic expansion of some order at $t=0$,
we can construct an almost coupled K\"ahler-Einstein metric in the following sense;
\begin{corollary}\label{Cor}
%Let $(X, \ke)$ be a Fano K\"ahler Einstein manifold.
Suppose the function $\mathcal{F}$ in Theoren\ref{main thm} has an asymptotic expansion $\mathcal{F}(t, \eta)=a_1(\eta)t+a_2(\eta)t^2+\cdots$ as $t\to 0$ with
$a_1(\eta)=a_2(\eta)=\cdots =a_m(\eta)=0$
for some $\eta\in U_0$ and for some positive integer $m$.
Then there exists $\e_0, C>0$ such that for any $i=1,2, \dots, N$ and any $t\in (0, \e_0)$, there exists a K\"ahler metric $\omega_i(t, \eta)$ in $\alpha_i+t[\eta_i]$ satisfying
$$\| 1- e^{f_i(t, \eta)} \|_{C^l(X, \mathbb{R})} \leq C t^{\frac{m+1}{2}},$$
where $(f_i(t,\eta))_{i=1}^N$ is the Ricci potential for $(\omega_i(t,\eta))_{i=1}^N$, and $l$ is some positive integer.
\end{corollary}

The same technique as in Theorem \ref{main thm} allows us to discuss the deformation of a coupled K\"ahler-Einstein metric on a Fano manifold when the complex structure varies.
Let $(X,J)$ be an $n$-dimensional Fano manifold with a complex structure admitting a coupled K\"ahker-Einstein metric $(\theta_i)_{i=1}^N$.
Consider a complex deformation $(J(t), (\theta_i(t))_{i=1}^N)$ of $(J, (\theta_i)_{i=1}^N)$ satisfying $\sum_{i=1}^N[\theta_i(t)]=2\pi c_1(X,J(t))$.
In general, the action of $\mathrm{Isom}_0(X,\theta_1)$ may not extend to $(J(t), (\theta_i(t))_{i=1}^N)$ for $t\neq 0$.
Based on a work of Rollin-Simanca-Tipler \cite{RST13} in the context of complex deformation theory of extremal K\"ahler metrics, 
we assume that a compact connected subgroup $G'\subset\mathrm{Isom}_0(X,\theta_1)$ acts holomorphically on $(J(t), (\theta_i(t))_{i=1}^N)$.
We denote $B_{G'}$ by the space of all such complex deformations.
Let $W_{G'}^{l+2,2}(X)$ be the subspace of $G'$-invariant real functions in the sobolev space $W^{l+2,2}(X)$.  
Define an operator $\mathbb{L} : (W^{l+2,2}(X))^N\to (W^{l,2}(X))^N$ as follows;
\begin{eqnarray}\label{Lapracian}
    \mathbb{L}(u_1,\dots,u_N)=
    \left(
    \begin{array}{c}
      \Delta_{\theta_1}u_1+\sum_{j=1}^N u_j-\int_X\sum_{j=1}^N u_j \frac{\omega_0^n}{\int_X\omega_0^n}  \\
      \vdots \\
      \Delta_{\theta_N}u_N+\sum_{j=1}^N u_j-\int_X\sum_{j=1}^N u_j \frac{\omega_0^n}{\int_X\omega_0^n} \\
    \end{array}
    \right),
\end{eqnarray}
where $\Delta_{\theta_i}$ is the negative Laplacian for $\theta_i$.
Let $\mathcal{H}_{\frak{z'}}$ be the space of all functions $(u_1,\dots, u_N)\in (C_{G'}^{\infty}(X;\mathbb{R}))^N$
such that $\int_Xu_i\theta_i^n=0$ for each $i$ and $\mathrm{grad}_{\theta_1}u_1=\cdots =\mathrm{grad}_{\theta_N}u_N=V$ for some holomorphic vector field $V$ on $(X,J)$ corresponding to an element in $\frak{z'}$, where $\frak{z'}$ is the center of $\mathrm{Lie}(G')$.
Then we have
\begin{theorem}\label{Thm2}
Let $(X, J, (\theta_i)_{i=1}^N)$ be a Fano manifold with a complex structure admitting a coupled K\"ahler-Einstein metric satisfing
\begin{equation}\label{nondeg}
\Ker\mathbb{L}\cap(W_{G'}^{l+2,2}(X))^N\subset \mathbb{R}^N\oplus\mathcal{H}_{\frak{z'}}.
\end{equation}
For any $(J(t), (\theta_i(t))_{i=1}^N) \in B_{G'}$, there exists $\e_0>0$ and a smooth function $\mathcal{G}: [0,\e_0)\to\mathbb{R}$ such that
if $\mathcal{G}(t)=0$ for some $t\in [0,\e_0)$, then there exists a coupled K\"ahler-Einstein metric for the decomposition $([\theta_i(t)])_{i=1}^N$.
\end{theorem}
The condition \eqref{nondeg} is an analogue of a condition introduced by Li \cite{HLi17} in the context of complex deformation theory for constant scalar curvature K\"ahler metrics.
According to \cite{HLi17}, Li's original condition coincides with the non-degeneracy condition for the relative Futaki invariant introduced by Rollin-Simanca-Tipler \cite{RST13}. 

It is able to prove a corresponding result as Corollary \ref{Cor} and an asymptotic expansion for $\mathcal{G}$ as in Theorem \ref{main thm} in the complex deformation setting.
Since these results will not be used in this paper, we however omit these proof.

This paper is organized as follows;
In Section \ref{deformcKE}, an operator to deform a coupled K\"ahler-Einstein metric by the implicit function theorem is introduced,
and the first part of Theorem \ref{main thm} and Corollary \ref{Cor} are proved.
In Section \ref{expansion}, an asymptotic expansion of the function $\mathcal{F}$ in Theorem \ref{main thm} is calculated to complete the proof of Theorem \ref{main thm}.
In Section \ref{cpxdeform}, the technique used in Section \ref{deformcKE} is applied to prove Theorem \ref{Thm2}.
\subsection*{Acknowledgment}
The author was partly supported by Grant-in-Aid for JSPS Fellowships for Young Scientists No.17J02783 and No.19J01482.

%%%%%%%%%%%%%%%%%%%%%%%%%%%%%%%%%%%%
\section{Deformation for coupled K\"ahler-Einstein metrics}\label{deformcKE}
In this section we prove Theorem \ref{main thm}.
Let $X$ be an $n$-dimensional Fano manifold admitting a coupled K\"ahler-Einstein metric $\ctheta$ for a decomposition $(\alpha_i)_{i=1}^N$, 
and $G$ be the identity component of the isometry group ${\rm Isom}_0(X, \theta_1)$.
Then ${\rm Isom}_0(X, \theta_1)=\cdots ={\rm Isom}_0(X, \theta_N)$ by \cite[Lemma 2.2]{Tak19}.
%Recall $$U=\Set{\eta=(\eta_1, \dots, \eta_N) | \eta_i \text{ is a $\mathbb{R}$-valued $\theta_i$-harmonic $(1,1)$-form s.t. $\sum_{i=1}^N [\eta_i]=0$}}.$$
Fix a K\"ahler metric $\omega_0$ such that $\mathrm{Ric}(\omega_0)=\sum_{i=1}^N\theta_i$.
The normalized volume form $\omega_0^n / \int_X\omega_0^n$ is equal to the others $\theta_1^n / \int_X\theta_1^n=\cdots = \theta_N^n / \int_X \theta_N^n$, 
which comes from the definition of the coupled K\"ahler-Einstein metric. 
For any $\eta=(\eta_1, \dots, \eta_N)\in U_0$,
note that every $\eta_i$ is automatically $G$-invariant.
Let $W_G^{l+2,2}(X)$ be the subspace of real $(l+2)$-th sobolev space $W^{l+2,2}(X)$, whose elements are $G$-invariant.
Note $W_G^{l+2,2}(X) \subset C^m(X;\mathbb{R})$ if $l+2>n+m$ by the sobolev embbeding theorem.
We can take a neighborhood $\mathcal{U}_{l+2}\subset (W_{G}^{l+2,2}(X))^N$ at the origin to assume that
there exists $\e>0$ such that 
$\omega_i(t, \phi_i) := \theta_i+ t\eta_i + \del\dbar\phi_i$ defines a K\"ahler metric 
for any $t\in[0,\e)$, any $\eta\in U_0$, any $(\phi_i)_{i=1}^N\in\mathcal{U}_{l+2}$ and each $i=1,2,\dots,N$.
%For small $t\in [0, \e)$, we define 
%\begin{equation}
%\mathcal{M}(t, \eta, l+2)=\prod_{i=1}^N\Set{\phi_i\in W_G^{l+2,2}(X) | \omega_i(t, \phi_i):=\theta_i+t\eta_i+\deldel\phi_i>0}
%\end{equation}
%as the space of $G$-invariant K\"ahler potentials for the K\"ahler metrics $(\theta_i+t\eta_i)_{i=1}^N$, 
%and take a neighborhood $\mathcal{U}_{l+2}\subset(W_G^{l+2,2}(X))^N$ of the origin.
%It is able to suppose $\mathcal{U}_{l+2} \subset \mathcal{M}(t, \eta, l+2)$ for any $t\in [0,\e)$ and for any $\eta \in U_0$.
For $\Phi=(\phi_i)_{i=1}^N\in\mathcal{U}_{l+2}$, we denote $(f_i(t, \Phi))_{i=1}^N$ as the Ricci potential for $(\omega_i(t, \phi_i))_{i=1}^N$.

In order to construct a coupled K\"ahler-Einstein metric for the decomposition $(\alpha_i+t[\eta_i])_{i=1}^N$, we consider the following operator 
$\mathbb{F}=(F_1, F_2, \dots , F_{2N}) : [0, \e)\times\mathcal{U}_{l+2}\to (W^{l,2}_G)^N \times\mathbb{R}^{N}$;
%In order to deform the coupled K\"ahler-Einstein metric $(\theta_i)_{i=1}^N$ we consider the following operator 
%$\mathbb{F}=(F_1, F_2, \dots , F_{2N}) : (-\e, \e)\times\mathcal{U}_l\to (W^{2,l-2}_G)^N \times\mathbb{R}^{N}$;

\begin{eqnarray}\label{operator}
F_k(t,\Phi)=
  \begin{cases}
    1- e^{f_i(t, \Phi)} \quad &(k=1,2, \dots ,N)  \\
     \log \frac{1}{\int_X\omega_0^n}\int_X e^{-\sum_{j=1}^N\phi_j}\omega_0^n \quad &(k=N+1) \\
     \int_X\phi_k\omega_0^n \quad &(k=N+2,N+3, \dots ,2N). \\
  \end{cases}
\end{eqnarray}
Note that $(\omega_i(t, \phi_i))_{i=1}^N$ defines a coupled K\"ahler-Einstein metric if and only if $\mathbb{F}(t, (\phi_i+c_i)_{i=1}^N)=0$ for some constants $c_i\in\mathbb{R}$, 
and note also $\mathbb{F}(0,0)=0$.

\begin{remark}
Hultgren-WittNystr\"om \cite{HN18} introduced another operator  to prove  Theorem \ref{HN}. See \cite{HN18} for more detail.
However it is technically natural to use our operator $\mathbb{F}$  from view points of the Futaki type invariant.
The operator $\mathbb{F}$ was inspired by a generalization for K\"ahler-Einstein metrics introduced by Mabuchi \cite{Ma01}
which is called the generalized K\"ahler-Einstein metric or the Mabuchi soliton in the literature. 
\end{remark}

Consider the equation $\delta_{\Phi}\mathbb{F}(0,0)=0$ for a variation $(\delta\phi_1, \dots , \delta\phi_N)\in T_{(0,0)}(\{0\}\times\mathcal{U}_{l+2})$ 
to apply the implicit function theorem, where $\delta_{\Phi}\mathbb{F}(0,0)$ stands for the derivative along $\Phi$-direction at $(t,\Phi)=(0,0)$.
%It is easy to see that $\delta_{\Phi}\mathbb{F}(0,0)=0$ 
This is equivalent to the following equations;
\begin{equation}\label{delF}
 \Delta_{\theta_i}\delta\phi_i+\sum_{j=1}^N\delta\phi_j=0 \quad\text{and}\quad \int_X\delta\phi_i\omega_0^n=0 \quad\text{for}\quad i=1,2,\dots,N,
\end{equation}
where $\Delta_{\theta_i}$ is the negative Laplacian for $\theta_i$.
To see this, we prove the following;
\begin{lemma}\label{variationoff}
The variation of the Ricci potential $f_i(t,\Phi)$ along $(\delta\phi_1, \dots , \delta\phi_N)\in T_0\mathcal{U}_{l+2}$ at $(t,\Phi)=(0,0)$ is
$$\delta_{\Phi}f_i= -\Delta_{\theta_i}\delta\phi_i -\sum_{j=1}^N\delta\phi_j + \int_X\sum_{\j=1}^N\delta\phi_j\frac{\omega_0^n}{\int_X\omega_0^n}.$$
\end{lemma}
\begin{proof}
The derivation of the first equation in \eqref{Ricci} shows
$\delta_{\Phi}f_i= -\Delta_{\theta_i}\delta\phi_i -\sum_{j=1}^N\delta\phi_j+C$ for some constant $C$.
The constant $C$ is equal to $\int_X\sum_{\j=1}^N\delta\phi_j\frac{\theta_i^n}{\int_X\theta_i^n}
=\int_X\sum_{\j=1}^N\delta\phi_j\frac{\omega_0^n}{\int_X\omega_0^n}$ by the derivation of the second equation in \eqref{Ricci}.
\end{proof}
If the above equation \eqref{delF} has only the trivial solution, the implicit function theorem can be applied.
However it has nontrivial solutions in general by the following result.
Recall an operator $\mathbb{L} : (W^{l+2,2}(X))^N\to (W^{l,2}(X))^N$ defined by.
\begin{eqnarray*}
    \mathbb{L}(u_1,\dots,u_N)=
    \left(
    \begin{array}{c}
      \Delta_{\theta_1}u_1+\sum_{j=1}^N u_j-\int_X\sum_{j=1}^N u_j \frac{\omega_0^n}{\int_X\omega_0^n}  \\
      \vdots \\
      \Delta_{\theta_N}u_N+\sum_{j=1}^N u_j-\int_X\sum_{j=1}^N u_j \frac{\omega_0^n}{\int_X\omega_0^n} \\
    \end{array}
    \right).
\end{eqnarray*}
\begin{lemma}{\rm (A specific situation in \cite[Proposition 2.4]{Tak19-2})}
The kernel $\Ker\mathbb{L}$ is equal to
\begin{equation*}
\Set{(u_1,\dots,u_N)\in (C^{\infty}(X;\mathbb{R}))^N | \mathrm{grad}_{\theta_1}u_1=\cdots =\mathrm{grad}_{\theta_N}u_N=:V \text{ and }\text{ $V$ is holomorphic}},
\end{equation*}
where $\mathrm{grad}_{\theta_i}u_i$ is a type $(1,0)$ gradient vector field on $X$ defined by $i_{(\mathrm{grad}_{\theta_i}u_i)}\theta_i=\sqrt{-1}\bar{\partial}u_i$.
\end{lemma}
%In particular a holomorphic vector field on $X$ corresponds to an element of the Kernel $\Ker\mathbb{L}$.
Therefore we modify the operator $\mathbb{F}$ to apply the implicit function theorem.
Let $\frak{g}$ be the Lie algebra of $G$.
Since $X$ is Fano, $\frak{g}$ is nothing but the ideal of killing vector fileds with zeros.
%Any element $\xi\in\frak{g}$ corresponds to the holomorphic vector field $J\xi+\sqrt{-1}\xi$.
Let $\frak{z}$ be the center of $\frak{g}$.
For any $G$-invariant K\"ahler metric $\omega$ and for any $\xi\in\frak{z}$, 
the holomorphic vector field $V=J\xi+\sqrt{-1}\xi$ defines a smooth real-valued G-invariant function $u$ satisfying 
$$i_V\omega=\sqrt{-1}\dbar u \quad\text{and}\quad \int_Xu\omega^n=0,$$
where $J$ is a fixed complex structure of $X$.
The function $u$ is called {\it the holomorphic potential} of $V$ with respect to $\omega$.
For the holomorphic potentials $u_i$ of $V$ with respect to $\theta_i$, 
we call $\bm{u}=(u_1, \dots, u_N)$ {\it the holomorphic potential vector} of $V$ 
with respect to the coupled K\"ahler-Einstein metric $(\theta_i)_{i=1}^N$.
Let $\mathcal{H}_{\frak{z}}$ be the space of holomorphic potential vectors corresponding to elements in $\frak{z}$
with respect to $(\theta_i)_{i=1}^N$,
endowed with the induced $L^2$-inner product 
$\lla \bm{u},\bm{v} \rra_{\omega_0}=\int_X \langle\bm{u}, \bm{v}\rangle\omega_0^n/\int_X\omega_0^n$ from $(W_G^{l+2,2}(X))^N$
where $\langle\bm{u}, \bm{v}\rangle$ is the pointwise inner product.
Note that $\mathbb{R}^N\oplus\mathcal{H}_{\frak{z}}$ is nothing but the space of $G$-invariant kernels $\Ker\mathbb{L}\cap (W_G^{l+2,2}(X))^N$.
%$\mathbb{L}: (W_G^{l+2,2}(X))^N \to (W_G^{l,2}(X))^N$.
Note also that the operator $\mathbb{L}$ is self-adjoint with respect to $\lla \cdot,\cdot \rra_{\omega_0}$ (See \cite[Proposition 2.4]{Tak19-2} for more details).
%We fix an orthonormal basis $\bm{u}_1, \dots, \bm{u}_d$ of $\mathbb{R}^N\oplus\mathcal{H}_{\omega_0}$ 
%with respect to $\lla \cdot,\cdot \rra_{\omega_0}$.
%Define a natural inner product in the sobolev space $(W_G^{l+2,2})^N$ as
%$\lla \bm{u},\bm{v} \rra_{\omega_0}=\int_X \langle\bm{u}, \bm{v}\rangle\omega_0^n/\int_X\omega_0^n$,
%where $\langle\bm{u}, \bm{v}\rangle$ is the pointwise inner product.
%Note that the operator $\mathbb{L}$ is self-adjoint with respect to $\lla \cdot,\cdot \rra_{\omega_0}$ (See \cite[Proposition 2.3]{Tak19-2} for more details).
By the inner product $\lla \cdot,\cdot \rra_{\omega_0}$, 
the sobolev space $(W_G^{l+2,2}(X))^N$ is decomposed as $\mathbb{R}^N\oplus\mathcal{H}_{\frak{z}}\oplus\mathcal{H}_{\frak{z},l+2}^{\perp}$, 
where $\mathcal{H}_{\frak{z},l+2}^{\perp}$ is the orthogonal complement.
%Fix an orthonormal basis $\{\bm{v}_1, \dots , \bm{v}_d\}$ of $\mathcal{H}_G$ with respect to the inner product $\lla \cdot,\cdot \rra_{\omega_0}$.
We fix an orthonormal basis $\bm{v}_1, \dots, \bm{v}_d$ of $\mathbb{R}^N\oplus\mathcal{H}_{\frak{z}}$ 
with respect to $\lla \cdot,\cdot \rra_{\omega_0}$.
Let $\pi^{\perp}_{\frak{z}} : (W_G^{l+2,2}(X))^N\to\mathcal{H}_{\frak{z}, l+2}^{\perp}$ be the orthogonal projection.
%that is, for $\bm{u}\in (W_G^{l+2,2}(X))^N$ we define
%$$\pi^{\perp}_{\frak{z}}(\bm{u})=\bm{u}-\sum_{p=1}^d\lla \bm{u},\bm{v}_p \rra_{\omega_0}\bm{v}_p.$$

In these setting, we define a modified operator 
$\tilde{\mathbb{F}}=(\tilde{F_1},\dots,\tilde{F}_{2N}) 
: [0,\e )\times (\mathcal{U}_{l+2}\cap\mathcal{H}_{\frak{z},l+2}^{\perp})\to\mathcal{H}_{\frak{z}, l}^{\perp}\times\mathbb{R}^N$ 
as follows; 
\begin{eqnarray}\label{modifiedop}
  \begin{cases}
   (\tilde{F}_1,\dots, \tilde{F}_N)&= \pi_{\frak{z}}^{\perp}(F_1,\dots, F_N) \\
   \tilde{F}_{N+i}&= F_{N+i} \quad (i=1,2, \dots,N).  \\
  \end{cases}
\end{eqnarray}
Then the equation $\delta_{\Phi}\tilde{\mathbb{F}}(0,0)=0$ for a variation 
$(\delta\phi_1, \dots , \delta\phi_N)\in T_{(0,0)}(\{0\}\times(\mathcal{U}_{l+2}\cap\mathcal{H}_{\frak{z},l+2}^{\perp}))$ is equivalent to
\begin{eqnarray}
  \begin{cases}
    \pi_{\frak{z}}^{\perp} \circ \mathbb{L}(\delta\phi_1, \dots , \delta\phi_N) =0  \\
    \int_X\delta\phi_i\omega_0^n=0 \quad\text{for}\quad i=1,2,\dots,N.  \\
   \end{cases}
\end{eqnarray}
This equation has the only solution $(\delta\phi_1,\dots,\delta\phi_N)=(0,\dots,0)$,
%by definition of the projection $\pi^{\perp}_{\frak{z}}$.
since $(\delta\phi_1,\dots,\delta\phi_N) \in \mathcal{H}_{l+2, \frak{z}}^{\perp}$ and $\mathbb{L}(\delta\phi_1,\dots,\delta\phi_N) \in \mathcal{H}_{l, \frak{z}}^{\perp}$.
Therefore, by the implicit function theorem, for small $t>0$ there exists
$(\phi_i(t,\eta))_{i=1}^N \in \mathcal{U}_{l+2} \cap \mathcal{H}_{\frak{z},l+2}^{\perp}$ 
such that $\tilde{\mathbb{F}}(t, (\phi_i(t,\eta))_{i=1}^N)=0$.
More precisely, we have 

\begin{lemma}\label{IFT}
There exists $\e_0>0$ such that for any $t\in [0,\e_0)$ and for any $\eta\in U_0$, we have 
a pair of K\"ahler potentials $(\phi_i(t,\eta))_{i=1}^N \in \mathcal{U}_{l+2} \cap \mathcal{H}_{\frak{z},l+2}^{\perp}$
and functions $c(t,\eta) := (c_p(t,\eta))_{p=1}^d : [0,\e_0)\times U_0 \to \mathbb{R}^d$
satisfying
 \begin{equation}\label{GKE}
\Bigl(1-e^{f_1(t,\eta)},\dots, 1-e^{f_N(t,\eta)} \Bigr)=\sum_{p=1}^d c_p(t,\eta)\bm{v}_p,
\end{equation}
where $(f_i(t,\eta))_{i=1}^N$ denotes the Ricci potential for the K\"ahler metrics $(\theta_i+t\eta_i +\deldel\phi_i)_{i=1}^N$.
Moreover there exists $C>0$ such that for any $t\in [0,\e_0)$, any $\eta\in U_0$ and each $i=1,2,\dots,N$,
$\| \phi_i(t,\eta) \|_{W_G^{l+2,2}} \leq C\e_0$
and $\| c(t,\eta) \|_{\rm Euc} := \{\sum_{p=1}^d c_p(t,\eta)^2\}^{1/2} \leq C\e_0$.
\end{lemma}

Now we define the function $\mathcal{F}$ as in Theorem \ref{main thm}.
Let $V(t,\eta)$ be the holomorphic vector field on $X$ 
corresponding to $\sum_{p=1}^d c_p(t,\eta)\bm{v}_p\in \mathbb{R}^N\oplus\mathcal{H}_{\frak{z}}$ in Lemma \ref{IFT}, that is,
$V(t,\eta)=\mathrm{grad}_{\theta_1}(1-e^{f_1(t,\eta)}) = \cdots = \mathrm{grad}_{\theta_N}(1-e^{f_N(t,\eta)})$.
Let $H_i(t,\eta)$ be the holomorphic potential for $V(t,\eta)$ with respect to the K\"ahler metric $\omega_i(t,\eta):=\theta_i+t\eta_i +\deldel\phi_i(t,\eta)$,
that is,
%it is defined by 
\begin{equation}\label{holopote}
i_{V(t,\eta)}\omega_i(t,\eta)=\sqrt{-1}\bar{\partial}H_i(t,\eta) \quad\text{and}\quad \int_XH_i(t,\eta)\omega_i(t,\eta)^n=0
\end{equation}
for each $i$.
Then we introduce a function $\mathcal{F} : [0, \e_0)\times U_0\to\mathbb{R}$ as follows;
$$\mathcal{F}(t,\eta)=\int_X\sum_{i=1}^NH_i(t,\eta)(1-e^{f_i(t,\eta)})\frac{\omega_i(t,\eta)^n}{\int_X\omega_i(t,\eta)^n}.$$
This function is nothing but the Futaki type invariant $\mathrm{Fut}_c(V(t,\eta))$ for the holomorphic vector field $V(t,\eta)$ 
with respect to the decomposition $([\omega_i(t,\eta)])_{i=1}^N$.

Now we prove the first part of Theorem \ref{main thm}.
\begin{theorem}\label{part1}
There exists $\e_0>0$ such that if $\eta\in U_0$ satisfies $\mathcal{F}(t,\eta)=0$ for some $t\in[0,\e_0)$,
then there exists a coupled K\"ahler-Einstein metric for the decomposition $(\alpha_i+t[\eta_i])_{i=1}^N$.
\end{theorem}

\begin{proof}
It suffice to show that the vector $c(t,\eta) = (c_p(t,\eta))_{p=1}^d$ in Lemma \ref{IFT} vanishes when $\mathcal{F}(t,\eta)=0$.
%Then the K\"ahler metrics $(\omega_i(t,\eta))_{i=1}^N=(\theta_i+t\eta_i +\deldel\phi_i)_{i=1}^N$ defines a coupled K\"ahler-Einstein metric for  
%$(\alpha_i+t[\eta_i])_{i=1}^N$.

We first claim that there exists a constant $C>0$ independent of $t$ and $\eta$ satisfying
\begin{equation}\label{compare}
\| 1-e^{f_i(t,\eta)}-H_i(t,\eta) \|_{L^2(\omega_0)}  \leq C\e_0 \| c(t,\eta) \|_{\mathrm{Euc}}.
\end{equation}
for each $i$ (Here the $L^2$-norm is defined with respect to the measure $\omega_0^n/\int_X\omega_0^n$).
Indeed, by definition of the holomorphic vector field $V(t,\eta)$ and by \eqref{holopote}, we have
\begin{eqnarray*}
\sqrt{-1}\bar{\partial}(1-e^{f_i(t,\eta)}-H_i(t,\eta)) 
&=&i_{V(t,\eta)}(t\eta_i+\deldel\phi_i(t,\eta)) \\
&=& \sum_{p=1}^d c_p(t,\eta) i_{V_p}(t\eta_i+\deldel\phi_i(t,\eta)),
\end{eqnarray*}
where $V_p$ denotes the gradient holomorphic vector field corresponding to $\bm{v}_p\in\mathbb{R}^N\oplus\mathcal{H}_{\frak{z}}$.
Since $\| \phi_i(t,\eta) \|_{W_G^{l+2,2}} \leq C\e_0$ for any $t\in[0,\e_0)$ by lemma \ref{IFT}, we have
\begin{eqnarray*}
|\Delta_{\omega_0}(1-e^{f_i(t,\eta)}-H_i(t,\eta))|
&=&\Bigl| \sum_{p=1}^dc_p(t,\eta) \mathrm{tr}_{\omega_0}\{\partial (i_{V_p}(t\eta_i+\deldel\phi_i(t,\eta)))\} \Bigr| \\
&\leq& C\e_0 \| c(t,\eta) \|_{\mathrm{Euc}}.
\end{eqnarray*}
Then the eigenvalue decomposition for $\Delta_{\omega_0}$ and the normalization condition in \eqref{holopote} shows the estimate \eqref{compare}.
%\begin{equation}\label{compare2}
%\| 1-e^{f_i(t,\eta)}-H_i(t,\eta) \|_{L^2(\theta_i)}\leq C\e_0 \| c(t,\eta) \|_{\mathrm{Euc}}.
%\end{equation}
%Thus the inequality \eqref{compare} follows from \eqref{compare2} and the coupled K\"ahler-Einstein condition $\theta_i^n/\int_X\theta_i^n=\omega_0^n/\int_X\omega_0^n$ for all $i$.

Now we estimate the norm $\| c(t,\eta) \|_{\mathrm{Euc}}$.
Since $\{\bm{v}_1, \dots, \bm{v}_d \}$ is an orthonomal basis of $\mathbb{R}^N\oplus\mathcal{H}_{\frak{z}}$ and since the equation \eqref{GKE}, we have
\begin{eqnarray*}
\| c(t,\eta) \|_{\mathrm{Euc}}^2 
&=& \lla \sum_{p=1}^d c_p(t,\eta)\bm{v}_p, \sum_{q=1}^d c_q(t,\eta)\bm{v}_q \rra_{\omega_0} \\
&=& \int_X \sum_{i=1}^N (1-e^{f_i(t,\eta)})^2 \frac{\omega_0^n}{\int_X\omega_0^n} \\
&\leq& C \int_X \sum_{i=1}^N (1-e^{f_i(t,\eta)})^2 \frac{\omega_i(t,\eta)^n}{\int_X\omega_i(t,\eta)^n}
\end{eqnarray*}
where $C>0$ is a constant independent of $t$ and $\eta$, and we used the estimate $\| \phi_i(t,\eta) \|_{W_G^{l+2,2}} \leq C\e_0$ in Lemma \ref{IFT}.
By the assumption $\mathcal{F}(t,\eta)=0$, the inequality \eqref{compare} and the Cauchy-Schwarz inequality, we thus have
\begin{eqnarray*}
\int_X \sum_{i=1}^N (1-e^{f_i(t,\eta)})^2 \frac{\omega_i(t,\eta)^n}{\int_X\omega_i(t,\eta)^n}
&=& \int_X \sum_{i=1}^N(1-e^{f_i(t,\eta)}-H_i(t,\eta)) (1-e^{f_i(t,\eta)}) \frac{\omega_i(t,\eta)^n}{\int_X\omega_i(t,\eta)^n} \\
&\leq& C \sum_{i=1}^N \|1-e^{f_i(t,\eta))}-H_i(t,\eta) \|_{L^2(\omega_0)} \| 1-e^{f_i(t,\eta))} \|_{L^2(\omega_0)} \\
&\leq& C\e_0 \| c(t,\eta) \|_{\mathrm{Euc}} \sqrt{N}\Bigl( \sum_{i=1}^N  \| 1-e^{f_i(t,\eta))} \|_{L^2(\omega_0)}^2 \Bigr)^{1/2} \\
&\leq& C\e_0 \| c(t,\eta) \|_{\mathrm{Euc}}^2,
\end{eqnarray*}
where each $C$'s is again a positive constant independent of $t$ and $\eta$.
Therefore if $\e_0>0$ is small enough then $c(t,\eta) =0$.
This completes the proof.
\end{proof}

Using the same technique as in the previous proof we prove Corollary \ref{Cor}.
%\section{Almost coupled K\"ahler-Einstein metrics}
%We use same notation as in the section \ref{deformcKE} to prove Corollary \ref{Cor}.
It follows from the assumption that $| \mathcal{F}(t,\eta) | \leq Ct^{m+1}$ for any $t\in [0,\e_0)$.
By the same calculation as the last estimate in the previous proof, the following holds;
\begin{eqnarray*}
\Bigl| \mathcal{F}(t,\eta)-\sum_{i=1}^N\int_X (1-e^{f_i(t,\eta)})^2 \frac{\omega_i^n(t,\eta)}{\int_X\omega_i^n(t,\eta)} \Bigr|
%&=&  \int_X \sum_{i=1}^N(1-e^{f_i(t,\eta)}-H_i(t,\eta)) (1-e^{f_i(t,\eta)}) \frac{\omega_i(t,\eta)^n}{\int_X\omega_i(t,\eta)^n} \\
%&\leq& C \sum_{i=1}^N \|1-e^{f_i(t,\eta))}-H_i(t,\eta) \|_{L^2(\omega_0)} \| 1-e^{f_i(t,\eta))} \|_{L^2(\omega_0)} \\
%&\leq& C\e_0 \| c(t,\eta) \|_{\mathrm{Euc}} \sqrt{N}\Bigl( \sum_{i=1}^N  \| 1-e^{f_i(t,\eta))} \|_{L^2(\omega_0)}^2 \Bigr)^{1/2} \\
&\leq& C\e_0 \| c(t,\eta) \|_{\mathrm{Euc}}^2.
\end{eqnarray*}
Since $\| \phi_i(t,\eta) \|_{W_G^{l+2,2}} \leq C\e_0$ by lemma \ref{IFT}, 
then $\sum_{i=1}^N\int_X (1-e^{f_i(t,\eta)})^2 \frac{\omega_i^n(t,\eta)}{\int_X\omega_i^n(t,\eta)} \leq C\| c(t,\eta) \|_{\mathrm{Euc}}^2$.
Thus $\| c(t,\eta)\|_{\mathrm{Euc}}^2 \leq C t^{m+1}$ after perhaps replacing the constant $\e_0$ with a smaller one.
In view of the equation \eqref{GKE},  we finally have 
$$\| 1-e^{f_i(t,\eta)} \|^2_{C^l(X)} \leq C\| c(t,\eta)\|_{\mathrm{Euc}}^2 \leq C t^{m+1}.$$
This completes the proof of Corrollary \ref{Cor}.

%%%%%%%%%%%%%%%%%%%%%%%%%%%%%%%%%%%%%%%%%%%%
\section{The asymptotic expansion of the function $\mathcal{F}$}\label{expansion}
We use same notation as in the previous section to prove the second part of Theorem \ref{main thm}.
Namely we determine the leading term of the asymptotic expansion of $\mathcal{F}(t,\eta)$ at $t=0$ 
under the technical assumption $\mathrm{tr}_{\theta_i}\eta_i=0$ for $i=1,2,\dots,N$. 
%This assumption is used only for the formula $\frac{d}{dt} \bigl|_{t=0} \int_X\omega_i^n(t,\eta)=0$.
%The order of the asymptotic expansion is independent of this assumption.
Let $h_{\eta}$ be a smooth function defined by $\sum_{j=1}^N\eta_j=\deldel h_{\eta}$ and $\int_Xh_{\eta}\omega_0^n=0$.
Define $\bm{h}_{\eta}=(h_{\eta}, \dots, h_{\eta}) \in (C^{\infty}_G(X;\mathbb{R}))^N$.
Let $\pi_{\frak{z}}$ be the $L^2$-projection from  $(W_G^{l,2})^N$ to $\mathbb{R}^N\oplus\mathcal{H}_{\frak{z}}$.
%, that is, $\pi_{\frak{z}} + \pi_{\frak{z}}^{\perp}=\mathrm{id}$.

\begin{proposition}\label{part2}
Suppose $\mathrm{tr}_{\theta_i}\eta_i=0$ for $i=1,2,\dots,N$.
Then $$\mathcal{F}(t, \eta)= t^2 \int_X |\pi_{\frak{z}}(\bm{h}_{\eta})|^2\frac{\omega_0^n}{\int_X\omega_0^n} + \mathcal{O}(t^3)\quad\text{as}\quad t\to 0.$$
\end{proposition}
\begin{remark}
It follows from the proof of Proposition \ref{part2} that the first derivative $\mathcal{F}^{\prime}(0,\eta)$ with respect to $t$ vanishes without the assumption $\mathrm{tr}_{\theta_i}\eta_i=0$. 
This assumption is used for making the second derivative $\mathcal{F}^{\prime\prime}(0,\eta)$ simple to analyze.
\end{remark}
\begin{proof}
It is easy to see $\mathcal{F}(0,\eta)$ and the first derivative $\mathcal{F}^{\prime}(0, \eta)$ with respect to $t$ vanishes
by $H_i(0,\eta)=0$ and $f_i(0,\eta)=0$.
%$\omega_i(0,\eta)=\theta_i$ and $\frac{d}{dt}\int_X\omega_i^n(t,\eta)=0$.
%one has $$\mathcal{F}^{\prime}(0, \eta)=-\sum_{i=1}^N\int_X H^{\prime}_i(0,\eta)\frac{\theta_i^n}{\int_X\theta_i^n}.$$
%Take the derivative of the equation of the normalization condition in \eqref{holopote} to obtain $\int_X H^{\prime}_i(0,\eta)\theta_i^n=0$.
We describe the second derivative $\mathcal{F}^{\prime \prime}(0, \eta)$ in terms of the initial data.
%Taking the second derivative of the equation of the normalization condition in \eqref{holopote}, one has
%$\int_X H^{\prime \prime}_i(0,\eta)\theta_i^n+2\int_X H^{\prime}_i(0,\eta)\frac{d}{dt}\omega_i^n(t,\eta)|_{t=0}=0$ to conclude
It is also easy to see
\begin{eqnarray*}
\mathcal{F}^{\prime \prime}(0, \eta)
&=&-2\sum_{i=1}^N\int_X H^{\prime}_i(0,\eta) f^{\prime}_i(0,\eta)\frac{\theta_i^n}{\int_X\theta_i^n} \\
&=&-2\sum_{i=1}^N\int_X H^{\prime}_i(0,\eta) f^{\prime}_i(0,\eta)\frac{\omega_0^n}{\int_X\omega_0^n}.
\end{eqnarray*}
Here we used the equalities $\omega_0^n / \int_X\omega_0^n = \theta_1^n / \int_X\theta_1^n=\cdots = \theta_N^n / \int_X \theta_N^n$. 
%which follows from the definition of the coupled K\"ahler-Einstein metric.

The derivative of the defining equation of the holomorphic potential $H_i(t,\eta)$ in \eqref{holopote} shows
\begin{eqnarray*}
\sqrt{-1}\bar{\partial}H_i^{\prime}(0,\eta) 
&=&i_{V^{\prime}(0,\eta)}\omega_i(0,\eta)+i_{V(0,\eta)}\omega_i^{\prime}(0,\eta) \\
&=&-\sqrt{-1}\bar{\partial}f_i^{\prime}(0,\eta).
\end{eqnarray*}
Then $H^{\prime}_i(0,\eta)=-f_i^{\prime}(0,\eta)$ by equations 
$\int_XH^{\prime}_i(0,\eta)\theta_i^n=\int_X-f_i^{\prime}(0,\eta)\theta_i^n=0$
which come from the derivative of the normalization conditions for $H_i(t,\eta)$ and $f_i(t,\eta)$.

On the other hand, the derivative of the defining equation of the Ricci potential $f_i(t,\eta)$ as in \eqref{Ricci} 
together with the assumption $\mathrm{tr}_{\theta_i}\eta_i=0$ shows
$$f^{\prime}_i(0,\eta)= -\Delta_{\theta_i}\phi^{\prime}_i(0,\eta) -h_{\eta} -\sum_{j=1}^N\phi^{\prime}_j(0,\eta) +C$$ for some constant $C$.
Since this constant $C$ is equal to 
$\int_X\sum_{\j=1}^N\delta\phi_j\frac{\theta_i^n}{\int_X\theta_i^n}
=\int_X\sum_{\j=1}^N\delta\phi_j\frac{\omega_0^n}{\int_X\omega_0^n}$
by normalization conditions for $f_i^{\prime}(0,\eta)$ and $h_{\eta}$, then
\begin{equation}\label{derivativeoff}
(f_1^{\prime}(0,\eta), \dots, f_N^{\prime}(0,\eta))
=
-\mathbb{L}(\phi_1^{\prime}(0,\eta),\dots, \phi_N^{\prime}(0,\eta))-\bm{h}_{\eta}.
\end{equation}

Since the equation \eqref{GKE} shows
$\pi_{\frak{z}}(f_1^{\prime}(0,\eta), \dots, f_N^{\prime}(0,\eta))=(f_1^{\prime}(0,\eta), \dots, f_N^{\prime}(0,\eta))$
and since $\mathbb{L}(\phi_1^{\prime}(0,\eta),\dots, \phi_N^{\prime}(0,\eta)) \in \mathcal{H}_{l, \frak{z'}}^{\perp}$,
therefore $(f_1^{\prime}(0,\eta), \dots, f_N^{\prime}(0,\eta))=-\pi_{\frak{z}}(\bm{h}_{\eta})$ by \eqref{derivativeoff}.
This completes the proof.
\end{proof}

In the above proposition, if the initial coupled K\"ahler-Einstein metric $\ctheta$ is trivial,
that is, if there exists positive constants $(\lambda_i)_{i=1}^N$ satisfying $\sum_i\lambda_i=1$ and $\theta_i=\lambda_i\ke$ for  all $i$,
then the coefficient $\int_X |\pi_{\frak{z}}(\bm{h}_{\eta})|^2\omega_0^n / \int_X\omega_0^n$ in the asymptotic expansion vanishes.
Indeed $\Delta_{\ke}h_{\eta}=\sum_i\mathrm{tr}_{\theta_i}\eta_i/\lambda_i=0$, and thus $\bm{h}_{\eta}=0$ by the normalization condition of $h_{\eta}$.
We show the following to end this section.
Define 
$$\bm{I}_{\eta}
=\Bigl( |\eta_1|^2_{\theta_1}-\int_X|\eta_1|^2_{\theta_1}\frac{\theta_1^n}{\int_X\theta_1^n}, 
\dots, 
    |\eta_N|^2_{\theta_N}-\int_X|\eta_N|^2_{\theta_N}\frac{\theta_N^n}{\int_X\theta_N^n} \Bigr)\in (C^{\infty}_G(X;\mathbb{R}))^N.$$
%and also write $\pi^{\perp}_{\ke}$ as the projection $\pi^{\perp}_{\omega_0}$.

\begin{proposition}\label{part3}
Suppose $\mathrm{tr}_{\theta_i}\eta_i=0$ for $i=1,2,\dots, N$. 
Suppose also that there exist a K\"ahler-Einstein metric $\ke$
and positive constants $(\lambda_i)_{i=1}^N$ satisfying $\sum_{j=1}^N\lambda_j=1$ and $\theta_i=\lambda_i\ke$ for $i=1,2,\dots,N$.
Then $$\mathcal{F}(t, \eta)= \frac{t^4}{4} \int_X |\pi_{\frak{z}}(\bm{I}_{\eta})|^2\frac{\ke^n}{\int_X\ke^n} + \mathcal{O}(t^5)\quad\text{as}\quad t\to 0.$$
\end{proposition}
\begin{proof}
It is easy to see the third derivative $\mathcal{F}^{(3)}(0,\eta)$ with respect to $t$ vanishes by formulas 
$H_i(0,\eta)=f_i(0,\eta)=0$ and $H_i^{\prime}(0,\eta)=f_i^{\prime}(0,\eta)=0$ which come from the proof of Proposition \ref{part2} and $\bm{h}_{\eta}=0$.
In the following we describe the forth derivative $\mathcal{F}^{(4)}(0,\eta)$ in terms of the initial data.
A direct calculation with the above formulas shows 
$$\mathcal{F}^{(4)}(0,\eta)=-6\sum_{i=1}^N\int_XH_i^{\prime\prime}(0,\eta)f_i^{\prime\prime}(0,\eta)\frac{\ke^n}{\int_X\ke^n}.$$

By the second derivative of the defining equation of $H_i(t,\eta)$,
\begin{eqnarray*}
\sqrt{-1}\bar{\partial}H_i^{\prime\prime}(0,\eta) 
&=&i_{V^{\prime\prime}(0,\eta)}\omega_i(0,\eta)+2i_{V^{\prime}(0,\eta)}\omega_i^{\prime}(0,\eta)+i_{V(0,\eta)}\omega_i^{\prime\prime}(0,\eta) \\
&=&-\sqrt{-1}\bar{\partial}f_i^{\prime\prime}(0,\eta).
\end{eqnarray*}
Then $H_i^{\prime\prime}(0,\eta) =-f_i^{\prime\prime}(0,\eta)$, 
since $\int_XH^{\prime\prime}_i(0,\eta)\theta_i^n=\int_X-f_i^{\prime\prime}(0,\eta)\theta_i^n=0$
given by the second derivative of normalization conditions for $H_i(t,\eta)$ and $f_i(t,\eta)$.

Also observe $\phi_i^{\prime}(0,\eta)=0$.
Indeed formulas $f_i^{\prime}(0,\eta)=0$ and $\bm{h}_{\eta}=0$ and the equation \eqref{derivativeoff} show
$\phi_i^{\prime}(0,\eta)$ is constant, 
and its constant is equals to $0$ by $\int_X\phi_i^{\prime}(0,\eta)\omega_0^n=0$ 
which come from the derivative of the condition 
$\tilde{F}_k(t,(\phi_i(t,\eta))_{i=1}^N)=0$ for $k=N+1, \dots, 2N$ in the modified operator \eqref{modifiedop}.
%Now we claim 
%\begin{equation}
%(f_1^{\prime\prime}(0,\eta),\dots,f_N^{\prime\prime}(0,\eta))=-\bm{I}_{\eta}.
%\end{equation}

The second derivative of the defining equation of $f_i(t,\eta)$ 
together with the formula $\phi_i^{\prime}(0,\eta)=0$ and the assumption $\mathrm{tr}_{\theta_i}\eta_i=0$ shows
\begin{eqnarray*}
f_i^{\prime\prime}(0,\eta)
&=& \frac{d}{dt}\Bigl|_{t=0} \Bigl( -\mathrm{tr}_{\omega_i(t,\eta)}\bigl(\eta_i+\deldel\phi_i^{\prime}(t,\eta)\bigr)-h_{\eta}-\sum_{j=1}^N\phi_j^{\prime}(t,\eta)  \Bigr) + C \\
&=& |\eta_i|_{\theta_i}^2 -\Delta_{\theta_i}\phi_i^{\prime\prime}(0,\eta)-\sum_{j=1}^N\phi_j^{\prime\prime}(0,\eta) +C
\end{eqnarray*}
for some constant $C$.
Then $C=\int_X\sum_{j=1}^N\phi_j^{\prime\prime}(0,\eta)\frac{\omega_0^n}{\int_X\omega_0^n}
-\int_X|\eta_1|^2_{\theta_1}\frac{\theta_1^n}{\int_X\theta_1^n}$
by the normalization $\int_Xf_i^{\prime\prime}(0,\eta)\theta_i^n=\int_Xf_i^{\prime\prime}(0,\eta)\omega_0^n=0$.
Thus we have
\begin{equation*}
(f_1^{\prime\prime}(0,\eta),\dots,f_N^{\prime\prime}(0,\eta))
=-\mathbb{L}(\phi_1^{\prime\prime}(0,\eta),\dots,\phi_N^{\prime\prime}(0,\eta))+\bm{I}_{\eta}.
\end{equation*}

In view of the equation \eqref{GKE},
$\pi_{\frak{z}}(f_1^{\prime\prime}(0,\eta), \dots, f_N^{\prime\prime}(0,\eta))=(f_1^{\prime\prime}(0,\eta), \dots, f_N^{\prime\prime}(0,\eta)).$
Therefore $(f_1^{\prime\prime}(0,\eta), \dots, f_N^{\prime\prime}(0,\eta))=\pi_{\frak{z}}(\bm{I}_{\eta}).$
This completes the proof.
\end{proof}

%%%%%%%%%%%%%%%%%%%%%%%%%%%%%%%%%%%%%%
%%%%%%%%%%%%%%%%%%%%%%%%%%%%%%%%%%%%%%%%%%%%
\section{Deformation for complex structure and coupled K\"ahler-Einstein metrics}\label{cpxdeform}
In this section we consider the deformation of a coupled K\"ahler-Einstein metric on a Fano manifold under the deformation of  the complex structure
by applying the technique used in Section \ref{deformcKE}.
Let $(X, J)$ be a Fano manifold with a complex structure admitting a coupled K\"ahler-Einstein metric $(\theta_i)_{i=1}^N$.
As in Section \ref{deformcKE} we fix a K\"ahler metric $\omega_0$ satisfying $\mathrm{Ric}(\omega_0)=\sum_{i=1}^N\theta_i$.
Consider a smooth family of complex structure $J(t)$ with $J(0)=J$.
Kodaira-Spencer \cite{Kobook} showed that 
there exists a smooth family of compatible K\"ahler metric $\theta_i(t)$ with $J(t)$ for small $t>0$ satisfying $\theta_i(0)=\theta_i$ .
For our purpose, we only consider smooth families $J(t)$ and $(\theta_i(t))_{i=1}^N$ satisfying $\sum_{i=1}^N [\theta_i(t)] = 2\pi c_1(X,J(t))$. 
In this paper, such pair $(J(t), (\theta_i(t))_{i=1}^N)$ is called complex deformation of $(J, (\theta_i)_{i=1}^N)$.
We ask whether there exists a coupled K\"ahler-Einstein metrics for the decomposition $([\theta_i(t)])_{i=1}^N$.

Let $G$ be the identity component of the isometry group of the K\"ahler metric $\theta_i$.
Recall the identity component of the automorphism group of $(X, J)$ is the complexification of $G$.
The action of $G$ may not extend to $(X, J(t))$ in general.
Based on the idea of Rollin-Simanca-Tipler \cite{RST13}, 
we assume that there exists a compact connected subgroup $G'$ of $G$ such that the action of $G'$ extends holomorphically on the complex deformation $(J(t), (\theta_i(t))_{i=1}^N)$.
Let $B_{G'}$ be the space of complex deformations of $(J, (\theta_i)_{i=1}^N)$ admitting a holomorphic $G'$-action. 
%As in the section \ref{deformcKE},
Let $\frak{g'}$ be the Lie algebra of $G'$, and
%Since $X$ is Fano, $\frak{g'}$ is nothing but the ideal of killing vector fileds with zeros.
%Any element $\xi\in\frak{g}$ corresponds to the holomorphic vector field $J\xi+\sqrt{-1}\xi$.
$\frak{z'}$ be the center of $\frak{g'}$.
%For any $G'$-invariant K\"ahler metric $\omega$ and for any $\xi\in\frak{z'}$, the holomorphic vector field $V=J\xi+\sqrt{-1}\xi$ defines a smooth real-valued G-invariant function $H$ satisfying 
%$$i_V\omega=\sqrt{-1}\dbar H \quad\text{and}\quad \int_XH\omega^n=0.$$
%The function $H$ is called the holomorphic potential of $V$ with respect to $\omega$.
%For the holomorphic potentials $u_i$ of $V$ with respect to $\theta_i$, we call $\bm{u}=(u_1, \dots, u_N)$ the holomorphic potential vector of $V$ with respect to the coupled K\"ahler-Einstein metric $(\theta_i)_{i=1}^N$.
As in Section \ref{deformcKE}, let $\mathcal{H}_{\frak{z}'} \subset (W_{G'}^{l+2,2}(X))^N$ be the space of holomorphic potential vectors corresponding to elements in $\frak{z}'$ with respect to $(\theta_i)_{i=1}^N$,
% endowed with the $L^2$-inner product 
%$\lla \bm{u},\bm{v} \rra_{\omega_0}=\int_X \langle\bm{u}, \bm{v}\rangle\omega_0^n/\int_X\omega_0^n$.
and fix an orthonormal basis $\bm{v}_1, \dots, \bm{v}_d$ of $\mathbb{R}^N\oplus\mathcal{H}_{\frak{z}'}$ 
with respect to $\lla \cdot,\cdot \rra_{\omega_0}$.
Take the orthogonal decomposition $(W_{G'}^{l+2,2}(X))^N=\mathbb{R}^N\oplus\mathcal{H}_{\frak{z}'}\oplus\mathcal{H}_{\frak{z}, l+2}^{\perp}$,
and define $\pi_{\frak{z'}}^{\perp}$ as the projection $(W_{G'}^{l+2,2}(X))^N \to \mathcal{H}_{\frak{z}, l+2}^{\perp}$.
%The following is the main result of this section.
%\begin{theorem}
%Let $(X, J, (\theta_i)_{i=1}^N)$ be a Fano manifold with a complex structure admitting a coupled K\"ahler-Einstein metric satisfing
%\begin{equation}\label{nondeg}
%\Ker\mathbb{L}\cap(W_{G'}^{l+2,2}(X))^N\subset \mathbb{R}^N\oplus\mathcal{H}_{\frak{z'}}.
%\end{equation}
%For any $(J(t), (\theta_i(t))_{i=1}^N) \in B_{G'}$, there exists $\e_0>0$ and a smooth function $\mathcal{G}: [0,\e_0)\to\mathbb{R}$ such that
%if $\mathcal{G}(t)=0$ for some $t\in [0,\e_0)$, then there exists a coupled K\"ahler-Einstein metric for the decomposition $([\theta_i])_{i=1}^N$.
%\end{theorem}

Fix $(J(t), (\theta_i(t))_{i=1}^N) \in B_{G'}$.
%Let $W_{G'}^{l+2,2}(X)$ be the subspace of $G'$-invariant real function in the sobolev space $W^{l+2,2}(X)$.
For a neighborhood $\mathcal{U}_{l+2}\subset (W_{G'}^{l+2,2}(X))^N$ at the origin,
it is able to  assume there exists $\e>0$ such that 
$\theta_i(t)+\del_t\dbar_t\phi_i$ defines a K\"ahler metric for any $t\in[0,\e)$, any $(\phi_i)_{i=1}^N\in\mathcal{U}_{l+2}$ and each $i$, 
where $\dbar_t:=\frac{1}{2}(d-\sqrt{-1}J(t)d)$ and $\partial_t$ is its complex conjugate.
For $\Phi=(\phi_i)_{i=1}^N\in\mathcal{U}_{l+2}$, 
we denote by $(f_i(t, \Phi))_{i=1}^N$ the Ricci potential for $(\theta_i(t)+\del_t\dbar_t\phi_i)_{i=1}^N$.
In order to construct a coupled K\"ahler-Einstein metric for the decomposition $([\theta_i(t)])_{i=1}^N$,
consider an operator 
$\tilde{\mathbb{F}}: [0, \e)\times(\mathcal{U}_{l+2} \cap \mathcal{H}_{\frak{z'}, l+2}^{\perp}) \to  \mathcal{H}_{\frak{z'}, l}^{\perp} \times\mathbb{R}^N$
defined by
\begin{eqnarray}
  \begin{cases}
   (\tilde{F}_1,\dots, \tilde{F}_N)&= \pi_{\frak{z'}}^{\perp}(F_1,\dots, F_N) \\
   \tilde{F}_{N+i}&= F_{N+i} \quad (i=1,2, \dots,N),  \\
  \end{cases}
\end{eqnarray}
where each $F_k$ is defined by the same manner as \eqref{operator}.
Then the implicit function theorem shows the following;
\begin{lemma}\label{IFT'}
Suppose $\Ker\mathbb{L}\cap(W_{G'}^{l+2,2}(X))^N\subset \mathbb{R}^N\oplus\mathcal{H}_{\frak{z'}}$.
For any $(J(t), (\theta_i(t))_{i=1}^N) \in B_{G'}$, there exists $\e_0>0$ such that for any $t\in [0,\e_0)$, 
we have K\"ahler potentials $(\phi_i(t))_{i=1}^N\in\mathcal{U}_{l+2} \cap \mathcal{H}_{\frak{z'}, l+2}^{\perp}$ 
and functions $c(t):=(c_1(t),\dots, c_d(t)): [0,\e_0)\to\mathbb{R}^d$ satisfying
 \begin{equation}\label{GKE'}
\Bigl(1-e^{f_1(t)},\dots, 1-e^{f_N(t)} \Bigr)=\sum_{p=1}^d c_p(t)\bm{v}_p,
\end{equation}
where $(f_i(t))_{i=1}^N$ denotes the Ricci potential for the K\"ahler metrics $(\theta_i +\del_t\dbar_t\phi_i(t))_{i=1}^N$.
Moreover there exists $C>0$ such that for each $i=1,2,\dots,N$ and for any $t\in [0,\e_0)$, we have
$\| \phi_i(t) \|_{W_{G'}^{l+2,2}} \leq C\e_0$
and $\| c(t) \|_{\rm Euc} \leq C\e_0$.
\end{lemma}
\begin{proof}
By the same calculation as in Lemma \ref{variationoff}, the equation $\delta_{\Phi}\tilde{\mathbb{F}}(0,0)=0$ for a variation $(\delta\phi_1, \dots , \delta\phi_N)\in T_{(0,0)}(\{0\}\times (\mathcal{U}_{l+2}\cap\mathcal{H}_{\frak{z'}, l+2}^{\perp}))$ is given by
\begin{eqnarray}
  \begin{cases}
    \pi_{\frak{z'}}^{\perp} \circ \mathbb{L}(\delta\phi_1, \dots , \delta\phi_N) =0  \\
    \int_X\delta\phi_i\omega_0^n=0 \quad\text{for}\quad i=1,2,\dots,N.  \\
   \end{cases}
\end{eqnarray}
%\begin{eqnarray}
%  \begin{cases}
%   (\delta_{\Phi}\tilde{F}_1,\dots, \delta_{\Phi}\tilde{F}_N)(\delta\phi_1, \dots , \delta\phi_N)
%   &= \pi_{\frak{z'}}^{\perp}\mathbb{L}(\delta\phi_1, \dots , \delta\phi_N) \\
%   \delta_{\Phi}\tilde{F}_{N+i}
%   &= F_{N+i} \quad (i=1,2, \dots,N),  \\
%  \end{cases}
%\end{eqnarray}
Therefore the linearized operator 
$\delta_{\Phi}\tilde{\mathbb{F}}(0,0): T_{(0,0)}(\{0\}\times(\mathcal{H}_{\frak{z'}, l+2}^{\perp}\cap\mathcal{U}_{l+2}))
\to T_{(0,0)}(\mathcal{H}_{\frak{z'},l}^{\perp}\times\mathbb{R}^N)$ is invertible 
if and only if the condition $\Ker\mathbb{L}\cap(W_{G'}^{l+2,2}(X))^N\subset \mathbb{R}^N\oplus\mathcal{H}_{\frak{z'}}$ is satisfied.
\end{proof}

Now we define the function $\mathcal{G}: [0,\e_0)\to \mathbb{R}$ in Theorem \ref{Thm2}. 
Under the assumption $\Ker\mathbb{L}\cap(W_{G'}^{l+2,2}(X))^N\subset \mathbb{R}^N\oplus\mathcal{H}_{\frak{z'}}$, 
we have $(\phi_i(t))_{i=1}^N\in\mathcal{U}_{l+2} \cap \mathcal{H}_{\frak{z'}, l+2}^{\perp}$
and  $c(t):=(c_1(t),\dots, c_d(t)): [0,\e_0)\to\mathbb{R}^d$ as in Lemma \ref{IFT'}.
Let $\xi_p$ be the killing vector field in $\frak{z}'$ corresponding to $\bm{v}_p\in\mathbb{R}^N\oplus\mathcal{H}_{\frak{z'}}$.
For $p=1,\dots d$, the vector field $V_p(t):=J(t)\xi_p+\sqrt{-1}\xi_p$ is holomorphic on $(X, J(t))$ since $(J(t), (\theta_i(t))_{i=1}^N) \in B_{G'}$.
We define $V(t):=\sum_{p=1}^d c_p(t)V_p(t)$. 
%where $(c_1(t),\dots, c_d(t))$ are functions in Lemma \ref{IFT'}.
The holomorphic potential $H_i(t)$ for $V(t)$ with respect to $\omega_i(t):=\theta_i(t)+\del_t\dbar_t\phi_i(t)$ is defined by
\begin{equation}\label{H(t)}
i_{V(t)}\omega_i(t)=\sqrt{-1}\dbar_t H_i(t) \quad\text{and}\quad \int_X H_i(t)\omega_i(t)^n=0.
\end{equation}
Then we define $\mathcal{G}: [0,\e_0)\to \mathbb{R}$ as follows;
$$\mathcal{G}(t)=\int_X\sum_{i=1}^N H_i(t)(1-e^{f_i(t)})\frac{\omega_i(t)^n}{\int_X\omega_i(t)^n},$$
where $(f_i(t))_{i=1}^N$ is the Ricci potential for $(\theta_i(t) +\del_t\dbar_t\phi_i(t))_{i=1}^N$.

Since Theorem \ref{Thm2} is proved by the same way as in the proof of Theorem \ref{part1} together with the following lemma, we omit the proof of it.
\begin{lemma}
If $(J(t), (\theta_i(t))_{i=1}^N) \in B_{G'}$ satisfies
\begin{equation}
\| J(t)-J \|_{C^1(X, \omega_0)} \leq C\e_0,
\end{equation}
then there exists $C'>0$ such that for any $t\in [0,\e_0)$ and each $i$, 
$$\| 1-e^{f_i(t)} -H_i(t) \|_{L^2(\omega_0)}\leq C' \e_0 \| c(t) \|_{\mathrm{Euc}}.$$
\end{lemma}
\begin{proof}
First we define $\hat{V}(t)=\sum_{p=1}^d c_p(t)V_p(0)$ as a holomorphic vector field on $(X, J)$.
In view of the equation \eqref{GKE'}, for each $i$, it satisfies
\begin{equation*}
i_{\hat{V}(t)}\theta_i=\sqrt{-1}\dbar_0 (1-e^{f_i(t)}).
\end{equation*}
Together with \eqref{H(t)}, we have
\begin{equation}\label{estimateA}
\sqrt{-1}\dbar_0 (1-e^{f_i(t)}-H_i(t))=i_{\hat{V}(t)}\theta_i-i_{V(t)}\omega_i(t)+\sqrt{-1}(\dbar_t-\dbar_0)H_i(t).
\end{equation}

The estimates $\| \omega_i(t)-\theta_i \|_{W^{l+2,2}}\leq C\e_0$ given in Lemma \ref{IFT'} and
\begin{eqnarray*}
\|(V_p(t)-V_p(0)) \|_{C^0} + \| \partial_0(V_p(t)-V_p(0)) \|_{C^0}
&=&  \|(J(t)-J)\xi_p \|_{C^0} + \| \partial_0(J(t)-J)\xi_p \|_{C^0} \\
&\leq& C\e_0
\end{eqnarray*}
shows
\begin{eqnarray}\label{estimateB}
\Bigl\|\partial_0(i_{\hat{V}(t)}\theta_i-i_{V(t)}\omega_i(t)) \Bigr\|_{C^0}
&=& \Bigl\| \sum_{p=1}^d c_p(t) \partial_0 (i_{V_p(0)}\theta_i-i_{V_p(t)}\omega_i(t)) \Bigr\|_{C^0} \\
%&\leq& \Bigl(\sum_{p=1}^d\Bigl\|\partial_0 (i_{V_p(0)}\theta_i-i_{V_p(t)}\omega_i(t)) \Bigr\|_{C^0(X)}^2 \Bigr)^{1/2}
%           \Bigl\| c(t) \Bigr\|_{\mathrm{Euc}} \nonumber \\
&\leq& \Bigl\| c(t) \Bigr\|_{\mathrm{Euc}} \Bigl\{ \sum_{p=1}^d \Bigl( \Bigl\| i_{\partial_0(V_p(0)-V_p(t))}\theta_i \Bigr\|_{C^0} \nonumber \\
&&\quad\quad\quad\quad\quad\quad  + \Bigl\| i_{\partial_0V_p(t)}(\theta_i-\omega_i(t)) \Bigr\|_{C^0}  \nonumber \\
&&\quad\quad\quad\quad\quad\quad  + \Bigl\| i_{(V_p(0)-V_p(t))}\partial_0\theta_i \Bigr\|_{C^0} \nonumber \\
&&\quad\quad\quad\quad\quad\quad  + \Bigl\| i_{V_p(t)}\partial_0(\theta_i-\omega_0(t)) \Bigr\|_{C^0}\Bigr)^2 \Bigl\}^{1/2} \nonumber \\
&\leq& C\e_0 \| c(t) \|_{\mathrm{Euc}}. \nonumber
\end{eqnarray}

We next estimate the holomorphic potential $H_i(t)$.
Since $$\Delta_{\omega_i(t)}H_i(t)=\mathrm{tr}_{\omega_i(t)}\del_t\dbar_tH_i(t)
=\sum_{p=1}^dc_p(t)\mathrm{tr}_{\omega_i(t)}\partial_t(i_{V_p(t)}\omega_i(t))$$
and since $\| \omega_i(t)-\theta_i \|_{C^{2,\alpha}}\leq C\e_0$ 
(if the exponent $l$ is taken sufficiently large in Lemma \ref{IFT'}),
then there exists $C>0$ independent $t$ such that $\| H_i(t) \|_{C^2}\leq C \| c(t) \|_{\mathrm{Euc}}$.
Thus
\begin{equation}\label{estimateC}
\Bigl| \partial_0(\dbar_t-\partial_0)H_i(t) \Bigr| = \frac{1}{2} \Bigl| \partial_0 (J_t-J) dH_i(t) \Bigr| \leq C\e_0 \| c(t) \|_{\mathrm{Euc}}.
\end{equation}
Therefore, by \eqref{estimateA}, \eqref{estimateB} and \eqref{estimateC}, we have
$| \Delta_{\omega_0}(1-e^{f_i(t)}-H_i(t)) | \leq C\e_0 \| c(t) \|_{\mathrm{Euc}},$
and the eigenvalue decomposition for $\Delta_{\omega_0}$ and the normalization conditions for $f_i(t)$ and $H_i(t)$ shows
$\| 1-e^{f_i(t)} -H_i(t) \|_{L^2(\omega_0)}\leq C' \e_0 \| c(t) \|_{\mathrm{Euc}}.$
This completes the proof.
\end{proof}
%%%%%%%%%%%%%%%%%%%%%%%%%%%%%%%%%%%%%%

%%%%%%%%%%%%%%%%%
\bigskip

%\address{
%Department of Applied Mathematics \\ 
%Fukuoka University \\
%Fukuoka  814-0180\\
%Japan
%}
%{satoshi.nakamura.r8@dc.tohoku.ac.jp}
\end{document}